\input louC.sty

\documentclass[a4paper]{amsart}
\usepackage{amsrefs}
\usepackage{amsmath}
\usepackage{mathrsfs}
\usepackage{amssymb}
\usepackage{xcolor}
\usepackage{amscd} 

\usepackage[top=2in, bottom=2in, left=1.55in, right=1.55in]{geometry}

\theoremstyle{plain}
 \newtheorem{thm}{\textbf{Theorem}}[section]
 \newtheorem{prop}[thm]{\textbf{Proposition}}
 \newtheorem{lem}[thm]{\textbf{Lemma}}
 \newtheorem{cor}[thm]{\textbf{Corollary}}
\theoremstyle{definition}
 \newtheorem{dfn}[thm]{\textbf{Definition}}
\theoremstyle{remark}
 \newtheorem{rem}[thm]{\textbf{Remark}}
 \numberwithin{equation}{section}

\newcommand\coveredby{\mathrel{\ooalign{$<$\cr
  \hidewidth\raise0.0ex\hbox{$\cdot\mkern2mu$}\cr}}}
\newcommand\covers{\mathrel{\ooalign{$>$\cr
  \hidewidth\raise0.0ex\hbox{$\cdot\mkern7mu$}\cr}}}

\renewcommand{\leq}{\leqslant}
\renewcommand{\geq}{\geqslant}
\renewcommand{\setminus}{\smallsetminus}

\title[q-analogue and symmetric function analogue of a C-S-V result]{A q-analogue and a symmetric function analogue of a result by Carlitz, Scoville and Vaughan}
\author{\bfseries Yifei Li}

\begin{document}

\begin{abstract}
We derive an equation that is analogous to a well-known symmetric function identity: $\sum_{i=0}^n(-1)^ie_ih_{n-i}=0$. Here the elementary symmetric function $e_i$ is the Frobenius characteristic of the representation of $\cS_i$ on the top homology of the subset lattice $B_i$, whereas our identity involves the representation of $\cS_n\times \cS_n$ on the Segre product of $B_n$ with itself. We then obtain a q-analogue of a polynomial identity given by Carlitz, Scoville and Vaughan through examining the Segre product of the subspace lattice $B_n(q)$ with itself. We recognize the connection between the Euler characteristic of the Segre product of $B_n(q)$ with itself and the representation on the Segre product of $B_n$ with itself by recovering our polynomial identity from specializing the identity on the representation of $\cS_i\times \cS_i$.
\end{abstract}

\maketitle
\vspace{18mm}
\setcounter{page}{1}
\thispagestyle{empty}

\section{Introduction} \label{Intro}

Consider the power series $f(z)={\displaystyle \sum_{n=0}^{\infty} (-1)^n\frac{z^n}{n!n!}}$ and define the numbers $\Go_0$, $\Go_1$, $\Go_2$,...by ${\displaystyle \frac{1}{f(z)}=\sum_{n=0}^{\infty}\Go_n\frac{z^n}{n!n!}}$. It follows quickly from the definition that 
\beq \label{CSV_intro}
\sum_{k=0}^n(-1)^k {n \choose k}^2 \Go_k=0.
\eeq
Given $\Gs\in \cS_n$ a permutation of $[n]$. A number $i\in [n-1]$ is called an \emph{ascent} of $\Gs$ if $\Gs (i)<\Gs (i+1)$. Carlitz, Scoville and Vaughan \cite{CSV} proved that the number $\Go_k$ in equation \eqref{CSV_intro} is the number of pairs of permutations of $\cS_k$ with no common ascent. For example, $\Go_2=3$: $(12,21)$, $(21,12)$, $(21,21)$. The Bessel function $J_0(z)$ is essentially $f(z^2)$. Carlitz, Scoville and Vaughan's result provided a combinatorial interpretation of the coefficient $\Go_k$ in the reciprocal Bessel function. 

Recall that $[n]_q:=q^{n-1}+q^{n-2}+...+1$ is the $q$-analogue of the natural number $n$ and ${\displaystyle [n]_q!:=\prod_{i=1}^{n}{[i]_q}}$. Then the $q$-analogue of ${\displaystyle {n \choose k}}$ is ${\displaystyle {n \brack k}_q:=\frac{[n]_q!}{[k]_q![n-k]_q!}}$. For a permutation $\Gs\in \cS_n$, the \emph{inversion statistic} is defined by $$inv(\Gs):=|\{(i,j):1\leq i<j\leq n \mbox{ and } \Gs(i)>\Gs(j)\}|.$$ In the paper we will prove the following $q$-analogue of Carlitz, Scoville and Vaughan's result. Let $\cD_n$ denote the set $\{(\Gs,\Go)\in \cS_n\times\cS_n \mbox{ }|\mbox{ }\Gs\mbox{ and } \Go \mbox{ have no common ascent}\}$. Then 
\beq\label{W_n(q)_intro}
\sum_{i=0}^{n}{{n\brack i}_q^2 (-1)^i W_i(q)}=0
\eeq
and $W_i(q)=\sum_{(\Gs,\Go)\in \cD_i}{q^{inv(\Gs)+inv(\Go)}}$. Put $F(z)={\displaystyle \sum_{n=0}^{\infty} (-1)^n\frac{z^n}{[n]_q![n]_q!}}$. The function $F\Big((\frac{z}{2(1-q)})^2\Big)$ is the $q$-Bessel function $J_0^{(1)}(z;q)$. The $q$-Bessel functions were first introduced by F. H. Jackson in 1905 and can be found in later literature (see Gasper and Rahman \cite{Gasper}). The number of descending maximal chains $W_n(q)$ satisfies ${\displaystyle \frac{1}{F(z)}= \sum_{n=0}^{\infty}W_n(q)\frac{z^n}{[n]_q![n]_q!}}$, giving the coefficients of the reciprocal $q$-Bessel function a combinatorial meaning. We obtained our results through examining the Segre product of the subspace lattice $B_n(q)$ with itself and the representation of $\cS_n\times \cS_n$ on top homology of Segre product of the subset lattice $B_n$ with itself. Let us review the definition of the Segre product poset. 

\begin{dfn} \label{Segre} (Bj\"orner and Welker, \cite{Segre_rees}) 
Segre products of posets: Let $f: P \longrightarrow S$ and $g:Q\longrightarrow S$ be poset maps. Let $P\circ_{f,g}Q$ be the induced subposet of the product poset $P\times Q$ consisting of the pairs $(p,q)\in P\times Q$ such that $f(p)=g(q).$
Let $S=\NN$, the set of the natural numbers. When $P$ is a pure poset with a rank function $f=\rho$, $P\circ_{\rho,g}Q$ is the Segre product of P and Q with respect to $g$, and we denote it by $P\circ_g Q$.   
\end{dfn}

One important fact about the Segre product poset, due to Bj\"orner and Welker \cite{Segre_rees}*{Theorem 1}, is the following: $P\circ_g Q$ is Cohen-Macaulay over field $k$ provided that both $P$ and $Q$ are Cohen-Macaulay over $k$ and $g(Q)\subset \rho(P)$ with $g$ being a strict poset map. For our purpose, $P$ and $Q$ are the same poset (either $B_n$ or $B_n(q)$), which is pure with a rank function $\rho$. Because of the Cohen-Macaulayness of $B_n\circ_{\rho}B_n$ and $B_n(q)\circ_{\rho}B_n(q)$, those posets are well behaved, which motivated us to investigate the representation of $\cS_n\times\cS_n$ on the homology of $B_n\circ_{\rho}B_n$ and related Whitney homology groups.

We observed that $W_n(q)$ is in fact the Euler characteristic of the Segre product poset $B_n(q)\circ_{\rho}B_n(q)$. Equation \eqref{W_n(q)_intro} gave us the first hint to an analogue of a well-known symmetric function identity: for $n\geq 1$, $\sum_{i=0}^n(-1)^ie_ih_{n-i}=0$. In section 2 we defined the product Frobenius characteristic map to serve as a useful tool in studying representations on $\cS_n\times\cS_n$ and proved a few properties of the map. Then using this tool we derived our symmetric function analogue in section 3 Theorem \ref{Stanley q-form}. In section 4, we present our initial finding, Theorem \ref{q-form}, the $q$-analogue to Carlitz, Scoville, and Vaughan's result. We show the relation of the Euler characteristic $W_n(q)$ with the representation of $\cS_n\times\cS_n$ on top homology of $B_n\circ_{\rho}B_n$ in Theorem \ref{sp_ch_pn}. We recognize that, combined with Theorem \ref{sp_ch_pn}, specializing the symmetric function analogue will recover the $q$-analogue to Carlitz, Scoville, and Vaughan's result. Finally in section 5, we provide an alternative proof for that result.

\section{The product Frobenius characteristic map}

The Frobenius characteristic map is often used to study a representation of the symmetric group. Here we define a product Frobenius characteristic map to help understand representations of $\cS_n\times \cS_n$. Let us consider two sets of variables $x=(x_1, x_2,...)$ and $y=(y_1, y_2,...)$. Let $\cR^n$ be the space of class functions on $\cS_n$ and $\cR=\oplus_n\cR^n$. Also $\Lambda(x)=\oplus_n\Lambda^n(x)$ and $\Lambda(y)=\oplus_n\Lambda^n(y)$ denote the the rings of symmetric functions in variables $(x_1, x_2,...)$ and $(y_1,y_2,...)$ respectively.

\begin{dfn} 
Let $\chi$ be a class function on $\cS_m\times\cS_n$. The product Frobenius characteristic map $ch: \cR\times \cR\longrightarrow\Lambda(x) \times\Lambda(y)$ is defined as: 
\beq \label{prod_ch_formula}
ch(\chi)=\sum_{(\mu,\lambda)\vdash (m,n)}z_{\mu}^{-1}z_{\lambda}^{-1}\chi_{(\mu,\lambda)}p_{\mu}(x)p_{\lambda}(y),
\eeq
where $\chi (\mu,\lambda)$ is the value of $\chi$ on the class $(\mu,\lambda)$ and $p_{\mu}$, $p_{\lambda}$ are power sum symmetric functions. The class $(\mu,\lambda)$ is indexed by a partition $\mu$ of $m$ and a partition $\lambda$ of $n$ that tell us the cycle shapes of elements of $\cS_m$ and $\cS_n$ respectively.
\end{dfn}

The irreducible representations of $\cS_n\times \cS_m$ are of the form $A^{(i)}\otimes B^{(j)}$, where $A^{(i)}$ and $B^{(j)}$ are each irreducibles of $\cS_n$ and $\cS_m$ respectively (Sagan \cite{Sagan}*{Theorem 1.11.3}). A representation $V$ of $\cS_n\times \cS_m$ can then be decomposed into a sum of irreducibles of $\cS_n\times\cS_m$.

\begin{prop} \label{prod_ch_dfn}
Let $V$, a representation of $\cS_m\times \cS_n$, have the following decomposition: $V={\displaystyle \bigoplus_{i,j}}c_{ij}A^{(i)}\otimes B^{(j)}$, where $A^{(i)}$'s and $B^{(j)}$'s are irreducible representations of $\cS_m$ and $\cS_n$ respectively and $c_{ij}\in\ZZ$. Then the product Frobenius characteristic of $V$ is
$$ch(V)=\sum_{i,j}c_{ij}ch^m(A^{(i)})(x)ch^n(B^{(j)})(y).$$ 
Here $ch^m(A^{(i)})(x)$ is the usual Frobenius characteristic of $A^{(i)}$ in the variable $x$ and $ch^n(B^{(j)})(y)$ is defined similarly.
\end{prop}

\begin{proof}
Let $\chi$ denote the character of $V$. Let $\chi^{Ai}$ and $\chi^{Bj}$ be the characters of $A^{(i)}$ and $B^{(j)}$ respectively. Then $\chi=\sum_{i,j}c_{ij}\chi^{Ai}\otimes\chi^{Bj}$. By Proposition \ref{prod_ch_dfn}, 
\begin{align*}
ch(V)&=\sum_{(\mu,\lambda)\vdash (m,n)}z_{\mu}^{-1}z_{\lambda}^{-1}\chi (\mu,\lambda)p_{\mu}(x)p_{\lambda}(y)\\
&=\sum_{(\mu,\lambda)\vdash (m,n)}z_{\mu}^{-1}z_{\lambda}^{-1}\sum_{i,j}c_{ij}\chi_{\mu}^{Ai}\chi_{\lambda}^{Bj}p_{\mu}(x)p_{\lambda}(y)\\
&=\sum_{i,j}c_{i,j}\Big(\sum_{\mu\vdash m}z_{\mu}^{-1}\chi^{Ai}_{\mu}p_{\mu}(x)\Big)\Big(\sum_{\lambda\vdash n}z_{\lambda}^{-1}\chi^{Bj}_{\lambda}p_{\lambda}(y)\Big)\\
& =\sum_{i,j}c_{i,j}ch^m(A^{(i)})(x)ch^n(B^{(j)})(y).\\
\end{align*}
The second equality comes from \cite{Sagan}*{Corollary 1.9.4}).
\end{proof}

Because the product Frobenius characteristic map is basically an extension of the usual characteristic map, we keep the notation $ch$ for product Frobenius characteristic map even though $ch$ was previously defined to be $\oplus_nch^n$ in various literature (Sagan \cite{Sagan}, Stanley \cite{ec2}).  The meaning of $ch$ will be clear in the given context.

Let $V$ and $W$ be representations of $\cS_m$ and $\cS_n$ with characters $f$ and $g$. $f\otimes g$ is the character of $V\otimes W$. Recall that the induction product $f\circ g$ is the induction of $f\otimes g$ from $\cS_m\times\cS_n$ to $\cS_{m+n}$. A fundamental property of the usual characteristic map is the following:
\begin{prop} \label{regular_chch} \emph{(Stanley \cite{ec2}*{Proposition 7.18.2})} The Frobenius characteristic map $ch: R\longrightarrow \Lambda$ is a bijective ring homomorphism, i.e., \emph{ch} is one-to-one and onto, and satisfies $$ch(f\circ g)=ch(f)ch(g).$$
\end{prop}
\begin{rem}
Given $V$ a representation of $\cS_m$ with character $f$ and $W$ a representation of $\cS_n$ with character $g$, let $V=\oplus_ia_iA^{(i)}$ and $W=\oplus_jb_jB^{(j)}$ be their decompositions of into irreducibles. It can be easily verified that the product Frobenius characteristic $ch(f\otimes g)=ch(f)(x)ch(g)(y)$. It is a symmetric function in $\Lambda^m\times\Lambda^n$, while the usual Frobenius characteristic $ch(f\circ g)=ch(f)(x)ch(g)(x)$ is a symmetric function in $\Lambda^{m+n}$. 
\end{rem}
We would like the product Frobenius characteristic map to be a homomorphism of rings as well. Given a $\cS_k\times\cS_l$-module $V$ with its character $\psi$ and a $\cS_m\times\cS_n$-module $W$ with its character $\phi$, $\psi\otimes\phi$ is the character of $V\otimes W$, which is a representation of $(\cS_k\times\cS_l)\times(\cS_m\times\cS_n)$. We want to produce a character of $\cS_{k+m}\times\cS_{l+n}$.

\begin{dfn}
For $\psi$ and $\phi$ as given above, we define the \emph{induction product} $\psi\circ\phi$ to be $\psi\otimes\phi\uparrow_{(\cS_k\times\cS_l)\times(\cS_m\times\cS_n)}^{\cS_{k+m}\times\cS_{l+n}}$.
\end{dfn}
\begin{prop} \label{prod_ch}
Assume given $\psi$ a class function on $\cS_k\times\cS_l$, and $\phi$ a class function on $\cS_m\times\cS_n$. The product Frobenius characteristic map $ch: R\times R\longrightarrow \Lambda(x)\times\Lambda(y)$ is a bijective ring homomorphism, i.e., \emph{ch} is one-to-one and onto, and satisfies $$ch(\psi\circ\phi)=ch(\psi)ch(\phi).$$
\end{prop}

Before proving this proposition, we need to first establish a lemma:
\begin{lem} \label{tensor_ind}
If $f$ is the character of a representation of $\cS_k\times\cS_m$ and $g$ is the character of a representation of $\cS_l\times\cS_n$, then 
$$f\otimes g\uparrow_{(\cS_k\times\cS_m)\times(\cS_l\times\cS_n)}^{\cS_{k+m}\times\cS_{l+n}}=f\uparrow_{\cS_k\times\cS_m}^{\cS_{k+m}}\otimes g\uparrow_{\cS_l\times\cS_n}^{\cS_{l+n}}.$$
\end{lem}

\begin{proof}
Suppose $\cS_k\times\cS_m<\cS_{k+m}$ has coset representatives $\{s_1,s_2,...,s_q\}$, $q=(k+m)!/(k!m!)$, and $\cS_l\times\cS_n<\cS_{l+n}$ has coset representatives $\{t_1,t_2,...,t_r\}$, $r=(l+n)!/(l!n!)$. Then $\{(s_i, t_j)\}$, $i\in [q]$, $j\in [r]$, is a set of coset representatives for $(\cS_k\times\cS_m)\times(\cS_l\times\cS_n)< \cS_{k+m}\times\cS_{l+n}$. For $(\Gs_{k+m},\Gs_{l+n})\in \cS_{k+m}\times\cS_{l+n}$,
\begin{align*}
f\otimes g\uparrow_{(\cS_k\times\cS_m)\times(\cS_l\times\cS_n)}^{\cS_{k+m}\times\cS_{l+n}}((\Gs_{k+m},\Gs_{l+n})) &=\sum_{i,j}f\otimes g\big((s_i^{-1},t_j^{-1})(\Gs_{k+m},\Gs_{l+n})(s_i,t_j)\big)\\
& = \sum_i{f(s_i^{-1}\Gs_{k+m}s_i)}\sum_j{g(t_j^{-1}\Gs_{l+n}t_j)}\\
&=f\uparrow_{\cS_k\times\cS_m}^{\cS_{k+m}}(\Gs_{k+m})g\uparrow_{\cS_l\times\cS_n}^{\cS_{l+n}}(\Gs_{l+n})\\
&=f\uparrow_{\cS_k\times\cS_m}^{\cS_{k+m}}\otimes g\uparrow_{\cS_l\times\cS_n}^{\cS_{l+n}}((\Gs_{k+m},\Gs_{l+n})).\\
\end{align*}
The second and fourth equalities come from \cite{Sagan}*{Theorem 1.11.2}.
\end{proof}

\begin{proof}[proof of Proposition \ref{prod_ch}] Suppose $\psi={\displaystyle \sum_{i,j}}a_{ij}\psi_k^{(i)}\otimes \psi_l^{(j)}$ with $\psi_k^{(i)}$'s and $\psi_l^{(j)}$'s are irreducible characters of representations of $\cS_k$ and $\cS_l$ respectively. Similarly, $\phi ={\displaystyle \sum_{u,v}}b_{uv}\phi_m^{(u)}\otimes\phi_n^{(v)}$. For any $\Gs_k\in\cS_k$, $\Gs_l\in\cS_l$, $\Go_m\in\cS_m$, and $\Go_n\in\cS_n$, by Theorem $1.11.2$ in \textit{the Symmetric Group} (Sagan \cite{Sagan}), we have
\begin{align*}
\psi\otimes\phi\big((\Gs_k,\Gs_l),(\Go_m,\Go_n)\big) & = \big(\sum_{i,j}a_{ij}\psi_k^{(i)}(\Gs_k)\psi_l^{(j)}(\Gs_l)\big)\big(\sum_{u,v}b_{uv}\phi_m^{(u)}(\Go_m)\phi_n^{(v)}(\Go_n)\big)\\
&=\sum_{i,j,u,v}a_{ij}b_{uv}\psi_k^{(i)}(\Gs_k)\phi_m^{(u)}(\Go_m)\psi_l^{(j)}(\Gs_l)\phi_n^{(v)}(\Go_n)\\
&=\sum_{i,j,u,v}a_{ij}b_{uv}(\psi_k^{(i)}\otimes\phi_m^{(u)})\otimes(\psi_l^{(j)}\otimes\phi_n^{(v)})(\Gs_k,\Go_m,\Gs_l,\Go_n).\\
\end{align*}

Thus, $\psi\otimes\phi=\sum_{i,j,u,v}a_{ij}b_{uv}(\psi_k^{(i)}\otimes\phi_m^{(u)})\otimes(\psi_l^{(j)}\otimes\phi_n^{(v)})$. So 

\begin{align*}
\psi\circ\phi & =\psi\otimes\phi\uparrow_{(\cS_k\times\cS_l)\times(\cS_m\times\cS_n)}^{\cS_{k+m}\times\cS_{l+n}}\\
& =\sum_{i,j,u,v}a_{ij}b_{uv}(\psi_k^{(i)}\otimes\phi_m^{(u)})\otimes(\psi_l^{(j)}\otimes\phi_n^{(v)})\uparrow_{\cS_k\times\cS_m\times\cS_l\times\cS_n}^{\cS_{k+m}\times\cS_{l+n}}\\
& =\sum_{i,j,u,v}a_{ij}b_{uv}(\psi_k^{(i)}\otimes\phi_m^{(u)})\uparrow_{\cS_k\times\cS_m}^{\cS_{k+m}}\otimes (\psi_l^{(j)}\otimes\phi_n^{(v)})\uparrow_{\cS_l\times\cS_n}^{\cS_{l+n}}\\ 
& = \sum_{i,j,u,v}a_{ij}b_{uv}(\psi_k^{(i)}\circ\phi_m^{(u)})\otimes (\psi_l^{(j)}\circ\phi_n^{(v)})\\
\end{align*}
by Lemma \ref{tensor_ind}. Now take the product Frobenius characteristic of both sides of the above equation. For clarity, we keep tracks of variables $x$ and $y$. By Proposition \ref{regular_chch} we get
\begin{align*}
ch(\psi\circ\phi)(x,y)&=\sum_{i,j,u,v}a_{ij}b_{uv}ch(\psi_k^{(i)}\circ\phi_m^{(u)})(x)ch(\psi_l^{(j)}\circ\phi_n^{(v)})(y)\\
&=\sum_{i,j,u,v}a_{ij}b_{uv}ch(\psi_k^{(i)})(x)ch(\phi_m^{(u)})(x)ch(\psi_l^{(j)})(y)ch(\phi_n^{(v)})(y)\\
&=\sum_{i,j}a_{ij}ch(\psi_k^{(i)})(x)ch(\psi_l^{(j)})(y)\sum_{u,v}b_{uv}ch(\phi_m^{(u)})(x)ch(\phi_n^{(v)})(y)\\
&=ch(\psi)(x,y)ch(\phi)(x,y)\\
\end{align*} 
\end{proof}

\section{A symmetric function analogue}

Using the product Frobenius characteristic map, we arrive at our main result regarding a representation of $\cS_n\times\cS_n$ on the Segre product of the subset lattice $B_n$ with itself. We derived an equation that is analogous to a well-known symmetric function identity, see Stanley \cite{ec2}*{equation (7.13)}:\\
for $n\geq 1$, $$\sum_{i=0}^n(-1)^ie_ih_{n-i}=0.$$ The thing to note is that the elementary symmetric function $e_i$ is the Frobenius characteristic of the representation of $\cS_i$ on the top homology of $B_i$. Our theorem will give the Segre product $B_n\circ_{\rho_n}B_n$ version of this identity.

\begin{thm} \label{Stanley q-form}
For the subset lattice $B_n$ with rank function $\rho_n$, let $P_n$ be the proper part of the Segre product poset $B_n\circ_{\rho_n}B_n$. Write $\cS_n$ for the symmetric group on $[n]$. The action of $\mathcal{S}_n \times \mathcal{S}_n$ induces a representation on the reduced top homology of $P_n$. Let $ch(\tiH_{n-2}(P_n))$ be the product Frobenious characteristic of this representation. Then 
\beq \label{hhch}
{\sum_{i=0}^{n}}{(-1)^ih_{n-i}(x)h_{n-i}(y)ch(\tiH_{i-2}(P_i))}=0,
\eeq
where $h_k$'s are the complete homogeneous symmetric functions.
\end{thm}

\begin{proof}
Let $Q$ be $P_{n}\cup{\hat{0}}$, which is Cohen-Macaulay. We consider the Whitney homology of $Q$. The action of $S_{n} \times S_{n}$ on $Q$ induces a representation of $S_{n} \times S_{n}$ on the reduced top homology of $Q$ and its Whitney homology groups. From the work of Sundaram on Whitney homology (Sundaram \cites{Sheila1, Sheila2}, Wachs \cite{Wachs_notes}), we know that 

$$\tiH_{n-2}(P_{n})\cong_{S_{n}\times S_{n}}{\displaystyle \bigoplus_{r=0}^{n-1}}(-1)^{n-1+r}\mbox{WH}_r(Q).$$ 

 Let $x$ be a rank $r$ element of $Q$. Then the stabilizer of $x$ is the young subgroup $(S_{r}\times S_{n-r})\times (S_{r}\times S_{n-r})$. Viewing the Whitney homology groups as representations,   

$$\mbox{WH}_r(Q)={\displaystyle \bigoplus_{x\in Q_r/(S_{n}\times S_{n})}}\tiH_{r-2}(\hat{0},x)\uparrow_{(S_{r}\times S_{n-r})\times (S_{r}\times S_{n-r})}^{S_{n} \times S_{n}},$$ 

where $Q_r$ is the set of rank $r$ elements in $Q$ and $Q_r/(S_{n} \times S_{n})$ is a set of orbit representatives (see Lecture $4.4$ in Wachs' \textit{Poset Topology} \cite{Wachs_notes}). The action of $S_{n} \times S_{n}$ on $Q_r$ is transitive. So the contribution of the $rth$ Whitney homology to $\tiH_{n-2}(P_{n})$ is the induced representation $\tiH_{r-2}(\hat{0},x)\uparrow_{(S_{r}\times S_{n-r})\times (S_{r}\times S_{n-r})}^{S_{n} \times S_{n}}$ for any $x$ in $Q_r$. The open interval $(\hat{0}, x)$ is isomorphic to the poset $P_r$. We then have

$$\tiH_{n-2}(P_{n})\cong_{S_{n}\times S_{n}}{\displaystyle \bigoplus_{r=0}^{n-1}}(-1)^{n-1+r}\tiH_{r-2}(P_{r})\uparrow_{(S_{r}\times S_{n-r})\times (S_{r}\times S_{n-r})}^{S_{n} \times S_{n}}.$$ 

Hence, \beq\label{ch_Whitney} ch(\tiH_{n-2}(P_{n}))={\displaystyle \sum_{r=0}^{n-1}}(-1)^{n-1+r}ch\big(\tiH_{r-2}(P_{r})\uparrow_{(S_{r}\times S_{n-r})\times (S_{r}\times S_{n-r})}^{S_{n} \times S_{n}}\big).\eeq 

Now we would like to relate $ch(\tiH_{r-2}(P_{r}))$ with the product Frobenius characteristic of the representation induced to $\mathcal{S}_n\times \mathcal{S}_n$. Let $\psi_r$ be the character of the $(S_r\times S_r)$-module $\tiH_{r-2}(P_{r})$. Write $1_{S_{n-r}\times S_{n-r}}$ for the character of the trivial representation of ${S_{n-r}\times S_{n-r}}$. When viewing $\tiH_{r-2}(P_{r})$ as a $(S_{r}\times S_{n-r})\times (S_{r}\times S_{n-r})$-module, its character equals $\psi_r\otimes 1_{S_{n-r}\times S_{n-r}}$ (Sagan, \cite{Sagan}*{Theorem 1.11.2}). Let $\psi_r\circ 1_{S_{n-r}\times S_{n-r}}$ denote the induction product of $\psi_r$ and $1_{S_{n-r}\times S_{n-r}}$. Then
\begin{align*}
\tiH_{r-2}(P_{r})\uparrow_{(S_{r}\times S_{n-r})\times (S_{r}\times S_{n-r})}^{S_{n} \times S_{n}} & =\psi_r\otimes 1_{S_{n-r}\times S_{n-r}}\uparrow_{(S_{r}\times S_{n-r})\times (S_{r}\times S_{n-r})}^{S_{n} \times S_{n}}\\
 & =\psi_r\circ 1_{S_{n-r}\times S_{n-r}}.
\end{align*}

It follows from Proposition \ref{prod_ch} that the product Frobenius characteristic $$ch(\psi_r\circ 1_{S_{n-r}\times S_{n-r}})=ch(\psi_r)ch(1_{S_{n-r}\times S_{n-r}}).$$ 

Thus, equation (\ref{ch_Whitney}) becomes 

\beq\label{ch_H_decomp}\begin{split}
ch(\tiH_{n-2}(P_{n})) & ={\displaystyle \sum_{r=0}^{n-1}}{(-1)^{n-1+r}ch(\tiH_{r-2}(P_r))ch(1_{S_{n-r}\times S_{n-r}})}\\
& = {\displaystyle \sum_{r=0}^{n-1}}{(-1)^{n-1+r}ch(\tiH_{r-2}(P_r))ch(1_{S_{n-r}})(x)ch(1_{S_{n-r}})(y)}.
\end{split}\eeq

It is known that the Frobenius characteristic of the trivial representation of $\mathcal{S}_n$ is $h_n$ (Stanley \cite{ec2}). Multiplying both sides of equation (\ref{ch_H_decomp}) by $(-1)^{n-1}$, we get

$$(-1)^{n-1}ch(\tiH_{n-2}(P_{n})) ={\displaystyle \sum_{r=0}^{n-1}}{(-1)^{r}ch(\tiH_{r-2}(P_r))h_{n-r}(x)h_{n-r}(y)}.$$
Finally, we conclude that $$\sum_{i=0}^{n}{(-1)^ih_{n-i}(x)h_{n-i}(y)ch(\tiH_{i-2}(P_i))}=0.$$

\end{proof}

Theorem \ref{Stanley q-form} was motivated by our initial findings regarding the $q$-analogue of equation \eqref{CSV_intro}. Once we formulated the specialization of $ch(\tiH_{i-2}(P_i))$, the $q$-analogue can be retrieved by taking the stable principal specialization of equation \eqref{hhch}, suggesting the truth of Theorem \ref{Stanley q-form}. 

\section{The $q$-analogue of a Carlitz, Scoville, and Vaughan's result}

Let $ps:\Lambda \longrightarrow \QQ [q]$ be the stable principal specialization. For a symmetric function $f$, $ps(f)$ is defined to be $f(1,q,q^2,...)$. A summary of the specializations of different bases for the symmetric functions can be found in Stanley's \textit{Enumerative Combinatorics vol. $2$} \cite{ec2}*{proposition $7.8.3$}. Consider a symmetric function $f$ in two sets of variables $(x_1,x_2,...)$ and $(y_1,y_2,...)$. We take the stable principal specialization of $f$ in each set of variables, that is substituting $(1,q,q^2,...)$ for both $(x_1,x_2,...)$ and $(y_1,y_2,...)$. The product Frobenius characteristic of the $\cS_i\times \cS_i$-modules $\tiH_{i-2}(P_i)$ are symmetric functions in two sets of variables. Then it is natural to ask what we can say about their specializations. It turns out that $ps(ch(\widetilde{H}_{n-2}(P_n)))$ has interesting relations with the Euler characteristic of the Segre product of the subspace lattice $B_n(q)$ with itself. Recall the definition of $B_n(q)$. Let $q$ be a prime power and $\mathbb{F}_q$ be the finite field of $q$ elements. Consider the $n$-dimensional linear vector space $\mathbb{F}_q^n$ and its subspaces, then $B_n(q)$ is the lattice of those subspaces ordered by inclusion.   

$B_n(q)$ is a geometric lattice whose every subspace is a span of its atoms (Stanley \cite{ec1}*{Example 3.10.2}). It is graded with a rank function $\rho(W)=$ the dimension of the subspace $W$. Recall that an \emph{edge labeling} of a bounded poset $P$ is a map $\lambda: \cE(P)\longrightarrow \Lambda$, where $\cE(P)$ is the set of edges of the covering relations $x\coveredby y$ of $P$ and $\Lambda$ is some poset. We can define a labeling of $B_n(q)$ in the following steps:

\textbf{1.} For a $1$-dimensional subspace $X$ of $\mathbb{F}_q^n$ (an atom of $B_n(q)$), let $x$ be a basis element of $X$.  Let $A$ denote the set of all atoms of the subspace lattice $B_n(q)$. We define a map 
$f$: $A \longrightarrow [n]$,  $f(X)=$ the index of the right most non-zero coordinate of $x$. For example, in $B_3(3)$, if $X=$ span of $\{<1,0,1>\}$, $Y=$ span of $\{<2,1,0>\}$, $f(X)=3$ and $f(Y)=2$.

\textbf{2}. For $X$ any subspace of $\mathbb{F}_q^n$, let $A(X)$ denote the set of atoms whose span is $X$. Let $Y$ be an element of $B_n(q)$ that covers $X$, then $A(Y)\supset A(X)$. Denote the set $f(A(Y))$\textbackslash $f(A(X))$ by $\mathcal L$. Let $\rho$ be the rank function of $B_n(q)$, which is defined by the dimensions of the subspaces. Since $\rho(Y)-\rho(X)=1$, the set $\mathcal L$ is a subset of $[n]$ and has exactly one element. This element will be the label of the edge $(X,\,Y)$. 

A maximal chain $c=(\hat{0}\coveredby x_1\coveredby \cdots \coveredby x_t\coveredby \hat{1})$ is \emph{increasing} if $\lambda (\hat{0},x_1)<\lambda(x_1,x_2)<\cdots<\lambda(x_t,\hat{1})$. A chain $c$ is then associated with a word $$\lambda (c)=\lambda (\hat{0},x_1)\lambda(x_1,x_2)\cdots\lambda(x_t,\hat{1}).$$ If $\lambda (c_1)$ lexicographically precedes $\lambda (c_2)$, we say that $c_1$ lexicographically precedes $c_2$ and we denote this by $c_1<_Lc_2$.
Next we recall the definition of an EL-labeling of a poset:
\begin{dfn} \label{EL} (Bj\"orner and Wachs \cite{BW2}*{Definition 2.1}) An edge labeling is called an \textit{EL-labeling} (edge lexicographical labeling) if for every interval $[x, y]$ in $P$,\\
(1) there is a unique increasing maximal chain $c$ in $[x, y]$, and\\
(2) $c<_Lc'$ for all other maximal chains $c'$ in $[x, y]$.
\end{dfn}

\begin{prop} \label{B_n(q)_EL} The labeling method described above is an EL-labeling on the subspace lattice $B_n(q)$.\end{prop}

\begin{proof} Edges in the same chain cannot take duplicate labels since $\FF_q^n$ is $n$-dimensional and any maximal chain must take all labels in $\{1,2,\dots,n\}$. Let $[X, Y]$ be a closed interval in $B_n(q)$. All maximal chains of $[X,\, Y]$ will take labels from the set $\mathcal L = f(A(Y))$\textbackslash $f(A(X))$.  Let $0<a_1<a_2<\cdots <a_l\leq n$ be all the elements of $\mathcal L$. For each $i$, $1\leq i \leq l$, there is a $1$-dimensional subspace $V_i$ of $\mathbb{F}_q^n$ with $f(V_i)=a_i$ and $V_i\vee X $, the join of $V_i$ and $X$, is in $[X, Y]$. The chain $c=(X\coveredby X\vee V_1\coveredby \cdots \coveredby X\vee V_1\vee V_2\vee \cdots \vee V_l=Y)$ is an increasing maximal chain of $[X, Y]$. Any other $1$-dimensional subspace $V_i'$ satisfying $f(V_i')=a_i$ and $X\vee V_1\vee \cdots V_{i-1}\vee V_i' \in [X,Y]$ must equal $X\vee V_1\vee \cdots \vee V_i$. Since there is only one way to arrange the $a_i$'s increasingly, $c$ satisfies definition \ref{EL} condition (1).

Suppose there is another maximal chain $c'=(X=W_0\coveredby W_1\coveredby\cdots\coveredby W_l=Y)$. Let $f(A(W_i))$\textbackslash $f(A(W_{i-1})) = {b_i}$ for all $i\in [l]$. Let $k$, $1\leq k\leq l$, be the smallest integer such that $b_k\neq a_k$. We know that $b_k$ must be in $\cL$ and $b_k\neq a_1, a_2,\dots , a_k$. Also $a_1, a_2,..., a_{k}$ are the smallest $k$ elements of $\cL$ arranged increasingly. It follows immediately that $b_k > a_k$. Therefore condition (2) in the above definition is also satisfied.

\end{proof}

Under this EL-labeling, each maximal chain of the subspace lattice $B_n(q)$ can then be identified with a permutation $\sigma$ of $S_n$. See section \ref{Intro} for the definition of inversion statistic $inv(\Gs)$.  

\begin{lem} \label{Chains same perm}
The number of maximal chains of $B_n(q)$ assigned label $\sigma \in S_n$ is $q^{inv(\sigma)}$.
\end{lem}

\begin{proof} Let $\sigma \in S_n$, for each $1$-dimensional subspace of $\mathbb{F}_q^n$, we can take the vector whose right most non-zero coordinate is $1$ as its basis element. For each $i\in [n-1]$, let $inv(\Gs(i))$ denote the number of pairs $(i,j)$ such that $1\leq i< j\leq n$ and $\sigma (i)>\sigma(j)$. The number of ways to choose an atom $W_1$ such that the edge $(0,W_1)$ takes label $\Gs(1)$ is clearly $q^{inv(\Gs(1))}$. Let $k\in [n]$, assume the chain $0\coveredby W_1\coveredby ...\coveredby W_{k-1}$, has label $\sigma (1) \sigma (2)...\sigma (k-1)$. For each $i\in [k-1]$, pick an atom $V_i\in A(W_{k-1})$ with $f(V_i)=\Gs(i)$ and $v_i$ the basis element of $V_i$. The vectors $v_1,\, v_2,\,...,\,v_{k-1}$ are linearly independent hence form a basis of $W_{k-1}$. In order for the edge $(W_{k-1}, W_k)$ to take label $\Gs(k)$, $W_k$ needs to be the join of $W_{k-1}$ and an atom whose basis element, call it $v_k$, has $1$ on the $\Gs(k)$th coordinate and all $0$'s after the $\Gs(k)$th coordinate. Then $v_1,\, v_2,\,...,\,v_k$ will form a basis for $W_k$. So we need to find the number of ways to choose a $v_k$ that each results in a distinct $W_k$.

The vector $e_k=<0,...,0,1,0,...0>$ who has $1$ on the $\Gs(k)$th coordinate and $0$ everywhere else certainly is a choice for $v_k$. For each $j$, such that $1\leq k<j\leq n$ and $\Gs(k)>\Gs(j)$, the $\Gs(j)$th coordinate appears before the $\Gs(k)$th. $W_{k-1}$ has no vectors whose right most non-zero coordinate is the $\Gs(j)$th, so varying the $\Gs(j)$th coordinate of $e_k$ will produce new vectors that are not in the span of $\{v_1,...,v_{k-1},e_k\}$. There are $inv(\Gs(k))$ choices for $j$, and for each $j$, there are $q$ choices for the value of the $j$th coordinate. Each choice will produce a different $v_k$ thus a different $W_k$. Therefore for any given chain $0\coveredby W_1\coveredby ...\coveredby W_{k-1}$ assigned label $\Gs(1)\Gs(2)...\Gs(k-1)$, there are $q^{inv(\Gs(k))}$ choices for $W_k$ such that the edge $(W_{k-1},W_k)$ takes label $\Gs(k)$. Hence the number of maximal chains assigned label $\Gs$ is $\prod^{i=n}_{i=1}{q^{inv(\Gs(i))}}=q^{\sum^{i=n}_{i=1}{inv(\Gs(i))}}=q^{inv(\Gs)}$.
\end{proof}
The following theorem from Bj\"orner and Wachs is essential to connecting the permutations of $\cS_n$ with the Segre product poset $B_n(q)\circ_{\rho} B_n(q)$:
\begin{thm} \label{EL_homology} (Bj\"orner and Wachs \cite{BW4}*{Theorem 4.1}, see also Wachs \cite{Wachs_notes}*{Theorem 3.2.4}). Suppose $P$ is a poset for which $\hat{P}$ admits an EL-labeling. Then $P$ has the homotopy type of a wedge of spheres, where the number of $i$-spheres is the number of decreasing maximal $(i+2)$-chains of $\hat{P}$. The decreasing maximal $(i+2)$-chains, with $\hat{0}$ and $\hat{1}$ removed, form a basis for homology $\tiH_i(P;\ZZ)$.
\end{thm}

Now consider the Segre product of $B_n(q)$ with itself. Denote the proper part of the Segre product by $P_n(q)$. Using the labeling of $B_n(q)$ described right before definition \ref{EL}, the Segre product of $B_n(q)$ with itself admits an edge-labeling in which the labels are ordered pairs from the poset $[n]\times [n]$. A label $(i,j)\in [n]\times [n] \leq (k,l)$ if and only if $i\leq k$ and $j\leq l$. It is easy to verify that this labeling of $B_n(q)\circ_{\rho} B_n(q)$ is an EL-labeling. The descending chains are labeled with pairs of permutations with no common ascent. Given a pair of permutations $(\Gs, \Go)$, the number of descending maximal chains assigned label $(\Gs, \Go)$ is $q^{inv(\Gs)}\cdot q^{inv(\Go)}$ from Lemma \ref{Chains same perm}. Recall that $\cD_n$ denotes the set of pairs of permutations $(\Gs,\Go)\in \cS_n\times\cS_n$ with no common ascent. Then we immediately arrive at the following proposition:
\begin{prop} 
Let $W_n(q)$ be the total number of descending maximal chains of $P_n(q)$. 
Then
$$W_n(q)=\sum_{(\Gs,\Go)\in\cD_n}{q^{(inv(\Gs)+inv(\Go))}}.$$
\end{prop}

\begin{rem}
The Segre product poset $B_n(q)\circ B_n(q)$ is the $q$-analogue of the Segre product poset $B_n\circ B_n$, agreeing with the formal definition of a $q$-analogue in R. Simion's paper \cite{Simion}. She showed that the $q$-analogue of an EL-shellable poset is also EL-shellable. This particular EL-labeling of $B_n(q)\circ B_n(q)$ provided intuition and a combinatorial interpretation for $W_n(q)$.
\end{rem}

\begin{thm} \label{q-form}
Let $P_n(q)$ be the proper part of the Segre product poset $B_n(q)\circ_{\rho} B_n(q)$. Let ${n\brack i}_q$ be the $q$-analogue of ${n\choose i}$ and $W_n(q)$ be the total number of descending maximal chains of $P_n(q)$. Then
\beq\label{W_n(q)}\sum_{i=0}^{i=n}{{n\brack i}_q^2 (-1)^i W_i(q)}=0.
\eeq
\end{thm}

\begin{proof} The poset $P_n(q)$ is pure. By Theorem \ref{EL_homology}, $P_n(q)$ has the homotopy type of a wedge of $(n-2)$-spheres, and its decreasing maximal $(n-2)$-chains form a basis of the reduced $(n-2)$-nd homology. Now we have $W_n(q)$ is the $n$th betti number of $P_n(q)$. The Euler-Poincar\'e formula \cite{Wachs_notes}*{Theorem 1.2.8} gives us 
\beq \label{EulerChar_P_n(q)} 
\widetilde{\chi}(\Delta(P_n(q)))=\sum_i {(-1)^i b_i(P_n(q))}
\eeq
where $b_i(P_n(q))$ is the $i$th betti number of $P_n(q)$. We know that the mobius number of a poset is the same as its reduced Euler Characteristic by the Philip Hall's theorem (Stanley \cite{ec1}*{Proposition 3.8.6}), so equation \ref{EulerChar_P_n(q)} now becomes 
$$\mu_{\widehat{P_n(q)}}(\hat{0},\hat{1})=\sum_i {(-1)^i b_i(P_n(q))}.$$
However, $P_n(q)$ is Cohen-Macaulay which means all reduced homology groups other than the top one vanish (Bj\"orner \cite{B_Shell_CM}). Thus $\mu_{\widehat{P_n(q)}}$ simplifies to $(-1)^n b_n(P_n(q))$, which is $(-1)^nW_n(q)$ in our set up. We then have 
\beq \label{Euler}
\mu_{\widehat{P_n(q)}}=(-1)^nW_n(q)=\widetilde{\chi}(\Delta(P_n(q))).
\eeq

On the other hand, by the definition of the m\"obius function, $$\mu(\hat{0},\hat{1})=-\sum_{\hat{0}\leq x<\hat{1}}{\mu(\hat{0}, x)}.$$ Each $x$ in $P_n(q)$ is a subspace of $\mathbb{F}_q^n\times \mathbb{F}_q^n$, which is the product of two $k$-dimensional subspaces $X_1, X_2$ of $\mathbb{F}_q^n$ for some $k$ with $0\leq k<n$. But the intervals $[\hat{0},X_1]$ and $[\hat{0}, X_2]$ are isomorphic to the poset $B_k(q)$, hence $\mu(\hat{0}, x)$ is just $\mu_{\widehat{P_k(q)}}(\hat{0}, \hat{1})$, where $P_k(q)= B_k(q)\circ_{\rho} B_k(q)\setminus \{\hat{0},\hat{1}\}$. The number of $k$-dimensional subspaces of $\mathbb{F}_q^n$ is ${n\brack k}_q$ (Stanley \cite{ec1}*{Proposition 1.7.2}), the $q$-analogue of $n\choose k$. Thus the number of distinct $x=(X_1,X_2)$ where $X_1$ and $X_2$ are $k$-dimensional subspaces is ${n\brack k}_q^2$. Therefore we have 
$$\mu_{\widehat{P_n(q)}}(\hat{0},\hat{1})=-\sum_{i=0}^{i=n-1}{{n\brack i}_q^2\mu_{\widehat{P_i(q)}}(\hat{0},\hat{1})}=-\sum_{i=0}^{i=n-1}{{n\brack i}_q^2 (-1)^iW_i(q)},$$ where $W_i(q)=\sum_{(\Gs,\Go)\in \mathcal{S}_i\times \mathcal{S}_i}{q^{inv(\Gs)+inv(\Go)}}$, which is summed over all pairs of permutations with no common ascent, is the number of descending maximal chains of the Segre product poset $B_i(q)\circ_{\rho_i}B_i(q)$.
\end{proof}

\begin{cor} \label{W_n(q)_is_Euler}
The Euler characteristic of the Segre product of the subspace lattice $B_n(q)\circ_{\rho} B_n(q)$ is $(-1)^nW_n(q)$.
\end{cor}
\begin{proof}
See equation \eqref{Euler} in the proof of theorem \ref{q-form}.
\end{proof} 

Recall that $P_n$ is the proper part of the Segre product of the subset lattice $B_n$ with itself. The product Frobenious characteristic of the $\cS_n\times \cS_n$-module $\tiH_{n-2}(P_n)$ has an innate connection with $W_n(q)$. The following theorem provides an equation that connects the stable principal specialization of $ch(\tiH_{n-2}(P_n))$ and the Euler characteristic $W_n(q)$.

\begin{thm} \label{sp_ch_pn}
Let $P_n$ be the proper part of Segre product of $B_n$ with itself and $\cS_n$ the symmetric group. The action of $\mathcal{S}_n \times \mathcal{S}_n$ induces a representation on the reduced top homology of $P_n$. Let $W_n(q)$ be the number of descending maximal chains of the Segre product of $B_n(q)$ with itself. For a symmetric function $f$ in two sets of variables $x=(x_1, x_2, \dots)$ and $y=(y_1,y_2,\dots)$, the stable principal specialization $ps(f)$ specializes both $x_i$ and $y_i$ to $q^{i-1}$. Then $$ps(ch(\widetilde{H}_{n-2}(P_n)))=\frac{W_n(q)}{\prod_{i=1}^{n}{(1-q^i)^2}},$$ where $ch(V)$ is the product Frobenious characteristic of a $\cS_n\times\cS_n$-module $V$.
\end{thm}

\begin{proof}
We will prove the proposition by induction. The base cases $n=2$ and $n=3$ can be verified by hand. 
$$ps(ch(\tiH_0(P_2)))=\frac{q^2+2q}{(1-q)^2(1-q^2)^2}=\frac{W_2(q)}{(1-q)^2(1-q^2)^2}$$ and 
$$ps(ch(\tiH_1(P_3)))=\frac{q^6+4q^5+6q^4+6q^3+2q^2}{(1-q)^2(1-q^2)^2(1-q^3)^2}=\frac{W_3(q)}{(1-q)^2(1-q^2)^2(1-q^3)^2}.$$ Assume that the statement is true for $P_i, i=1,...,n-1$. Now let us consider the reduced top homology of $P_{n}$. Equation (\ref{hhch}) gives us a way to express $ch(\tiH_{n-2}(P_n))$ in terms of Frobenius characteristic of smaller posets. That is

\beq \label{ch}
ch(\tiH_{n-2}(P_n))= \sum_{i=0}^{n-1}(-1)^{n-1+i}h_{n-i}(x)h_{n-i}(y)ch(\tiH_{i-2}(P_i))
\eeq

Then we take the stable principal specialization of both sides of equation (\ref{ch}). We know from Stanley's \textit{Enumerative Combinatorics vol. 2} that $ps(h_{n})=\prod_{i=1}^{n}{\frac{1}{1-q^i}}$ \cite{ec2}. It follows from our induction hypothesis that

\beq \label{ps_to_W_n(q)}
\begin{split}
ps(ch(\tiH_{n-2}(P_n))) & = \sum_{i=0}^{n-1}(-1)^{n-1+i}ps(ch(\tiH_{i-2}(P_i)))\prod_{j=1}^{n-i}{\frac{1}{(1-q^j)^2}}\\
& = \sum_{i=0}^{n-1}(-1)^{n-1+i}\frac{W_i(q)}{\prod_{k=1}^{i}{(1-q^k)^2}}\prod_{j=1}^{n-i}{\frac{1}{(1-q^j)^2}}\\
& = \frac{1}{\prod_{k=1}^{n}{(1-q^k)^2}}\cdot \sum_{i=0}^{n-1}{(-1)^{n-1+i}W_i(q)\frac{\prod_{j=i+1}^{n}{(1-q^j)^2}}{\prod_{j=1}^{n-i}{(1-q^j)^2}}}\\
& = \frac{1}{\prod_{k=1}^{n}{(1-q^k)^2}}\cdot \sum_{i=0}^{n-1}{(-1)^{n-1+i}W_i(q){n \brack i}_q^2}.
\end{split}
\eeq

Finally, using the identity involving the Euler characteristic $W_n(q)$ given in theorem \ref{q-form}, we obtain $$ps(ch(\tiH_{n}(P_{n})))=\frac{W_{n}(q)}{\prod_{j=1}^{n}{(1-q^j)^2}}.$$
\end{proof}

\section{Carlitz, Scoville and Vaughan's result and its alternative proof}

In Carlitz, Scoville and Vaughan's paper `Enumeration of pairs of permutations' \cite{CSV}, they gave the coefficients $\Go_k$ of the reciprocal of the Bessel function $J_0(z)$ a combinatorial explanation. They showed that $\omega_k$ is the number of pairs of $k$-permutations with no common ascent. When letting $q=1$ in our $q$-analogue \eqref{W_n(q)}, the subspaces of $\FF_q^n$ simply become subsets of $\{1,2,...,n\}$, and $W_n(1)=\sum_{(\Gs, \Go)\in \cS_n\times \cS_n} 1^{inv(\Gs)+inv(\Go)}$, where $(\Gs, \Go)$ is a pair of permutations with no common ascent, is in fact $\omega_n$. Hence we obtained the above result from Carlitz, Scoville and Vaughan. 

The proof of theorem \ref{q-form} can also be easily adapted to an alternative proof of Carlitz, Scoville and Vaughan's result \eqref{CSV_intro} by changing $B_n(q)$ to $B_n$, using $P_n$ instead of $P_n(q)$ to denote the Segre product, and recognizing that the intervals in the alternating sum for the M\"obius number of $\widehat{P_n}$ are isomorphic to smaller subset lattices $B_i$'s. Carlitz, Scoville and Vaughan's proof in \cite{CSV} included general cases where occurrences of common ascent are allowed. Our proof provides a less technical approach by utilizing Bj\"{o}rner and Wach's work on shellability and poset homology \cite{BW4}.  

\section{Acknowledgements}

The author would like to thank Washington University in St. Louis Department of Mathematics and Statistics for their support. She is especially grateful to John Shareshian for his insightful advice and valuable comments. 

\bibliography{References-Yifei}

\end{document}